\documentclass[11pt,a4paper,oneside]{amsart}
\newcounter{commentcounter}

\usepackage{amsmath} % Lots of maths functionality
\usepackage{amssymb} % Maths symbols
\usepackage{amsthm} % Maths environments: \begin{proof}, etc.
\usepackage{stmaryrd} % [[ brackets
\usepackage[english]{babel} % Language and hyphenation.
\usepackage[font=small,justification=centering]{caption} % More flexibility for captioning figures
\usepackage[nodayofweek]{datetime}
\usepackage[shortlabels]{enumitem} % Change enumeration labelling with \begin{enumerate}[a)] etc.
\usepackage[T1]{fontenc} % For font encoding, to allow accents, copy & paste, inequality signs, etc. to all work nicely.
\usepackage[utf8]{inputenc} % To be loaded after fontenc, also for encoding.
\usepackage{ifthen} % For Dani's \begin{com} environment and \numberedtheorem
\usepackage{mathabx} % Contains \pnest symbol and more. Clashes with accents for \ring
\usepackage{mathtools} % Uses amsmath, fixes quirks and adds functionality.
\usepackage[dvipsnames]{xcolor} % Allow links to have colour. Needs to be before hyperref and tikz-cd
\usepackage[pdftex,  colorlinks=true]{hyperref} % Makes references and citations into links
    \hypersetup{urlcolor=RoyalBlue, linkcolor=RoyalBlue,  citecolor=black}
\usepackage{setspace} % Allows \onehalfspacing etc. for changing gaps between lines
\usepackage{tikz-cd} % Commutative diagrams
\usepackage{xfrac} % Nicer quotients 
\usepackage[capitalize]{cleveref} % use \Cref{} for instead of X~\ref{} 
\makeatletter
\@namedef{subjclassname@1991}{Mathematical subject classification 1991}
\@namedef{subjclassname@2000}{Mathematical subject classification 2000}
\@namedef{subjclassname@2010}{Mathematical subject classification 2010}
\@namedef{subjclassname@2020}{Mathematical subject classification 2020}
\makeatother

\newtheorem{thm}{Theorem}[section]
\newtheorem{lemma}[thm]{Lemma}

\newtheorem{prop}[thm]{Proposition}
\newtheorem{conjecture}[thm]{Conjecture}

\newtheorem{question}[thm]{Question}

\newtheorem{thmx}{Theorem}

\theoremstyle{definition}
\newtheorem{defn}[thm]{Definition}
\newtheorem{remark}[thm]{Remark}

\theoremstyle{plain}

    \newtheoremstyle{TheoremNum}
        {\topsep}{\topsep} %%% space between body and thm
        {\itshape} %%% Thm body font
        {-0.25cm} %%% Indent amount (empty = no indent)
        {\bfseries} %%% Thm head font
        {.} %%% Punctuation after thm head
        { }  %%% Space after thm head
        {\thmname{#1}\thmnote{ \bfseries #3}}%%% Thm head spec
    \theoremstyle{TheoremNum}
    \newtheorem{duplicate}{}

\newcommand*{\claimproofname}{My proof}

\DeclareMathOperator{\Aut}{\mathrm{Aut}}
\DeclareMathOperator{\Out}{\mathrm{Out}}

\DeclareMathOperator{\Stab}{\mathrm{Stab}}

\DeclareMathOperator{\MCG}{\mathrm{MCG}}

\DeclareMathOperator{\st}{\mathrm{st}}
\DeclareMathOperator{\lk}{\mathrm{lk}}

\newcommand{\calf}{{\mathcal{F}}}

\newcommand{\calk}{{\mathcal{K}}}

\newcommand{\calo}{{\mathcal{O}}}
\newcommand{\calp}{{\mathcal{P}}}

\def\Z{\mathbb{Z}}

\newcommand{\NN}{\mathbb{N}}
\newcommand{\ZZ}{\mathbb{Z}}

\newcommand{\RR}{\mathbb{R}}

\newcommand{\KK}{\mathbb{K}}

\usepackage{tikz}
\usetikzlibrary{arrows,quotes}
\tikzstyle{blackNode}=[fill=black, draw=black, shape=circle]

\title[Homology growth of mapping tori]{Homology growth of polynomially growing mapping tori}

\author{Naomi Andrew}
\author{Yassine Guerch}
\author{Sam Hughes}
\author{Monika Kudlinska}
\address[Naomi Andrew]{Mathematical Institute, Andrew Wiles Building, Observatory Quarter, University of Oxford, Oxford OX2 6GG, United Kingdom}
\address[Sam Hughes]{Mathematisches Institut, Rheinische Friedrich-Wilhelms-Universität Bonn, Endenicher Allee 60, Bonn, 53115, Germany}
\address[Monika Kudlinska]{Emmanuel College, St Andrew's Street, Cambridge CB2 3AP, United Kingdom}
\address[Yassine Guerch]{Laboratoire de Mathématiques Nicolas Oresme, CNRS, UMR 6139,
University of Caen, Normandy,
Esplanade de la Paix, CS 14032, Caen Cedex 5, France}
\email{naomi.andrew@maths.ox.ac.uk}
\email{yassine.guerch@unicaen.fr}
\email{sam.hughes.maths@gmail.com}
\email{m.kudlinska@dpmms.cam.ac.uk}
\date{\today}

\subjclass[2020]{Primary 20J05; Secondary 20E05, 20E08, 20E26, 20F28, 57M07}

\begin{document}
\begin{abstract}
We prove that residually finite mapping tori of polynomially growing automorphisms of hyperbolic groups, groups hyperbolic relative to finitely many virtually polycyclic groups, right-angled Artin groups (when the automorphism is untwisted), and right-angled Coxeter groups have the cheap rebuilding property of Abert, Bergeron, Fraczyk, and Gaboriau.  In particular, their torsion homology growth vanishes for every Farber sequence in every degree.
\end{abstract}
\maketitle

\section{Introduction}

Let $\Gamma$ be a residually finite group of type $\mathrm{F}$. By L\"{u}ck's celebrated approximation theorem, the $i^{th}$ $\ell^2$-Betti number $b_i^{(2)}(\Gamma)$ of $\Gamma$ is a measure of the growth of the $i^{th}$ homology of $\Gamma$ with rational coefficients \cite{Lueck1994}. More precisely, if $(\Gamma_k)_{k \in \mathbb{N}}$ is a descending sequence of finite index normal subgroups of $\Gamma$
such that $\bigcap_{k \in \mathbb{N}} \Gamma_k = 1$, then 
\[b_i^{(2)}(\Gamma) = \lim_{k \to \infty}\frac{\mathrm{dim}_{\mathbb{Q}}H_i(\Gamma_k, \mathbb{Q})}{[\Gamma \colon \Gamma_k]}.\]
The $\ell^2$-Betti numbers are important group invariants which have found many applications in topology and group theory (see \cite{Lueck2002} and the references therein for a comprehensive account). It is thus natural to study the growth of other homology groups associated to $\Gamma$, as well as the growth of the torsion part $|H_i(\Gamma_k, \Z)_{\rm tors}|$ of the homology.  

This paper is concerned with the growth of the mod-$p$ Betti numbers of certain groups $\Gamma$, as well as the \emph{homology torsion growth}, which is defined to be
\[t_j(\Gamma; \Gamma_k) = \limsup_{k\to \infty} \frac{\log | H_j(\Gamma_k,\mathbb Z)_{\rm tors}|}{[\Gamma:\Gamma_k]}, \]
where $(\Gamma_k)_{k \in \mathbb{N}}$ is a Farber sequence of $\Gamma$. 

These invariants have been computed in right-angled Artin groups and certain graph products \cite{AvramidiOkunSchreve2021,OkunSchreve2021} (see also \cite{FisherHughesLeary2022} where non-vanishing is proven for certain Bestvina--Brady groups). Moreover, there exist results proving vanishing of homology growth in various instances \cite{BergeronVentktesh2013,Sauer2016,KarKrophollerNikolov2017,AbertBergeronFraczykGaboriau2021}. We also mention the work of Bader--Gelander--Sauer which gives an upper bound for the homology torsion of negatively curved Riemannian manifolds of dimension at least four in terms of their volume \cite{BaderGelanderSauer2020}. 

\medskip

The main aim of this paper is to calculate homology growth for certain classes of group extensions of the form $G \rtimes_{\phi} \Z$ where $\phi \in \Aut(G)$ is \emph{polynomially growing} (see \cref{sec:growth} for the precise definition). We prove:

\begin{thmx}\label{main}
    Let $\Gamma$ be a group isomorphic to one of
    \begin{itemize}
        \item $G\rtimes_\phi \Z$ with $G$ residually finite and hyperbolic;
        \item $G\rtimes_\phi\Z$ with $G$ residually finite and hyperbolic relative to a finite collection of virtually polycyclic groups;
        \item $A_L\rtimes_\phi \Z$ where $A_L$ is a right-angled Artin group and $\phi \in \mathrm{Aut}(A_L)$ is untwisted (see \cref{sec RA*Gs}); or
        \item $W_L\rtimes_\phi \Z$ where $W_L$ is a right-angled Coxeter group.
    \end{itemize}
    If $\phi$ is polynomially growing, then for every Farber sequence $(\Gamma_k)_{k \in \mathbb N}$ of $\Gamma$, every $j\geq0$ and every field $\KK$, we have
\[\lim_{k\to \infty} \frac{\dim_{\mathbb{K}} H_j (\Gamma_k, \mathbb{K})}{[\Gamma : \Gamma_k]} = 0 \quad \mbox{and} \quad \lim_{k\to \infty} \frac{\log | H_j(\Gamma_k,\mathbb Z)_{\rm tors}|}{[\Gamma:\Gamma_k]}=0.\]
\end{thmx}

We note that in the case where $G$ is a finite rank free group and $\phi \in \mathrm{Aut}(G)$ a polynomially-growing automorphism, the result was already known by previous work of the first, third and fourth author \cite{AndrewHughesKudlinska2022}. 

\medskip

The key tool in proving \cref{main} is the \emph{cheap rebuilding property} developed by Abert--Bergeron--Fraczyk--Gaboriau in \cite{AbertBergeronFraczykGaboriau2021}, which implies vanishing of homology torsion growth in residually finite groups. Crucially for us, a group $\Gamma$ has the cheap rebuilding property whenever it acts co-compactly on a contractible CW-complex with stabilisers which satisfy the cheap rebuilding property (see \cref{sec:CRP} for more details, and specifically \cref{thm:Crprebuild}). 

We prove that the groups considered in \cref{main} all satisfy the cheap rebuilding property. We do so by finding $\phi$-invariant splittings of $G$, which induce splittings of the extension $G \rtimes_{\phi} \Z$. In the one-ended relatively hyperbolic case, we deploy the theory of JSJ decompositions \cite{guirardellevitt2017jsj} following Guirardel--Levitt. In the case of right-angled Artin and Coxeter groups, we use the work of Fioravanti on coarse median preserving automorphisms of those groups \cite{Fioravanti22}.

For the case when the group $G$ is infinitely ended, we prove a combination theorem for the cheap rebuilding property of the mapping torus $G \rtimes_{\phi} \Z$ of a polynomially-growing automorphism $\phi \in \mathrm{Aut}(G)$: 

\medskip

\begin{duplicate}[\Cref{Thm:crpfreeproduct}]
    Let $G=G_1 \ast \ldots \ast G_k \ast F_N$ be a free product of residually finite groups. Fix $\alpha \in \NN$. Let $\phi$ be a polynomially-growing automorphism of $G$ which preserves the conjugacy classes of the factors $G_i$. Suppose that for every $i\in \{1,\ldots,k\}$, the group $G_i \rtimes_{\phi|_{G_i}} \ZZ$ has the cheap $\alpha$-rebuilding property. Then the group $G \rtimes_\phi \ZZ$ has the cheap $\alpha$-rebuilding property.
\end{duplicate}

\medskip

We note that \cref{Thm:crpfreeproduct} recovers the main theorem from \cite{AndrewHughesKudlinska2022} on the cheap rebuilding property for mapping tori of polynomially-growing outer automorphisms of finite rank free groups.

The key observation in the proof of \Cref{Thm:crpfreeproduct} is that if a polynomially-growing automorphism preserves a free product decomposition, then after possibly passing to a power it preserves a \emph{sporadic} free factor system (see \cref{Prop:existenceBStree}). Hence, it preserves a $G$-tree $T$ which is the Bass--Serre tree corresponding to the sporadic free factor system.

\subsection{The \texorpdfstring{$\ell^2$}{l2}-torsion conjecture}

The \emph{integral torsion} $\rho^{\Z}(\Gamma)$ of a group $\Gamma$ is defined to be the sum of the homology torsion gradients,
\[\rho^{\Z}(\Gamma) := \sum_{j \geq 0} t_j(\Gamma; \Gamma_k).\]  An important conjecture of L\"{u}ck relates the integral torsion $\rho^\Z(\Gamma)$ with its $\ell^2$-torsion $\rho^{(2)}(\Gamma)$:
\begin{conjecture}[{\cite[Conjecture~1.11(3)]{Lueck2013}}] Let $\Gamma$ be an infinite residually finite $\ell^2$-acyclic group of type $\mathrm{VF}$. Then, $\rho^{(2)}(\Gamma) = \rho^{\Z}(\Gamma)$. \end{conjecture} The conjecture is known to hold for some classes of groups, including amenable groups by the work of Kar--Kropholler--Nikolov \cite{KarKrophollerNikolov2017} and Li--Thom \cite{LiThom2014}, and fundamental groups of closed aspherical manifolds which admit a non-trivial $S^1$-action, or contain a nontrivial elementary amenable normal subgroup, by the work of L\"{u}ck \cite{Lueck2013}. Okun--Schreve showed that the conjecture is also true in the case of right-angled Artin groups, where the torsion does not vanish in general \cite{OkunSchreve2021}.

The work of Clay \cite{Clay2017}, combined with the results mentioned above, confirms the conjecture  in the case of mapping tori of polynomially-growing outer automorphisms of finite rank free groups. Hence, our work leads to the following natural question:

\begin{question}
    Let $\Gamma$ be the mapping torus of a polynomially-growing automorphism $\phi \in \mathrm{Aut}(G)$ as in \cref{main}. Is it true that the $\ell^2$-torsion $\rho^{(2)}(\Gamma)$ of $\Gamma$ vanishes?
\end{question}

\subsection*{Acknowledgements}
This work has received funding from the European Research Council (ERC) under the European Union's Horizon 2020 research and innovation programme (Grant agreement No. 850930). YG was supported by the LABEX MILYON of Université de Lyon. MK was supported by an Engineering and Physical Sciences Research Council studentship (Project reference 2422910).

YG thanks Damien Gaboriau for numerous very helpful discussions regarding the cheap rebuilding property. All the authors thank the referee for numerous useful comments which led to a significant improvement in the exposition of this paper.
\section{Background}

\subsection{Growth of automorphisms}\label{sec:growth}

Fix a finite generating set $S$ of a group $G$. For any $g \in G$, let $\| [g] \|_S$ denote the length of the shortest word in the conjugacy class $[g]$ of $g$, and $|g|_S$ the length of the shortest word in the generators $S$ representing the element $g$ in $G$. An outer automorphism $\Phi \in \mathrm{Out}(G)$ is said to \emph{grow polynomially}, if for every conjugacy class $c$ in $G$, there exists some integer $d \geq 0$ and real number $C > 0$ such that for all $n \in \mathbb{N}$, \[ \|\Phi^n(c)\|_S \leq Cn^d + C.\]
An automorphism $\phi \in \Aut(G)$ is said to \emph{grow polynomially}, if for every element $g \in G$, there exists some integer $d \geq 0$ and $C >0$ such that for all $n \in \NN$, \[|\phi^n(g)|_S  \leq Cn^d + C.\]
Note that for any two finite generating sets $S$ and $S'$, the corresponding word metrics on $G$ are bi-Lipschitz equivalent. In particular, the definitions of growth are independent of the choice of a finite generating set.

\begin{remark}
    Many naturally-occurring groups, including torsion-free hyperbolic groups and abelian groups, experience a growth rate dichotomy: for every $\phi \in \mathrm{Aut}(G)$, each $g \in G$ admits either polynomial or exponential growth under the iterations of $\phi$ (see \cite[Theorem~1.1]{Coulon2022}). In the case of free groups, Levitt~\cite{Levitt2009} gave complete classification of all possible types of growth for automorphisms of free groups. However, as shown by Coulon~\cite{Coulon2022}, there exist groups with automorphisms which exhibit more exotic types of growth.
\end{remark}

\subsection{Free products, free factor systems and the graph of free factors}

Let $G_1, \ldots, G_k$ be a finite collection of non-trivial finitely generated groups, let $F_N$ be a free group of rank $N$ and let $$G=G_1 \ast \ldots \ast G_k \ast F_N.$$ We denote by $\calf$ the set of conjugacy classes of the groups $G_i$ with $i \in \{1,\ldots,k\}$. We refer to the pair $(G,\calf)$ as a \emph{free product}.

\begin{defn}\label{Defn:sporadic}
The pair $(G,\calf)$ is a \emph{sporadic free product} if one of the following holds:
\begin{enumerate}
    \item we have $k=0$ and $G=\ZZ$;
    \item we have $k=1$ and $G=G_1$ or $G=G_1 \ast \ZZ$;
    \item we have $k=2$ and $G=G_1 \ast G_2$.
\end{enumerate}
Otherwise, the pair $(G,\calf)$ is a \emph{nonsporadic free product}.
\end{defn}

Given a free product $(G,\calf)$, an element $g \in G$ is \emph{peripheral} if there exists $[A]\in \calf$ with $g \in A$. Otherwise, we say that $g$ is \emph{nonperipheral}.

A \emph{free factor system of $(G,\calf)$} is a finite collection $\calf'=\{[A_1],\ldots, [A_\ell]\}$ of conjugacy classes of non-trivial finitely generated subgroups of $G$ such that: 
\begin{enumerate}
    \item for every $i \in \{1,\ldots,k\}$, the group $G_i$ is contained in some subgroup $A$ of $G$ with $[A] \in \calf'$;
    \item there exists a subgroup $B$ of $G$ such that $G=A_1 \ast \ldots \ast A_\ell \ast B$.
\end{enumerate} A free factor system of $(G,\calf)$ is \emph{proper} if it is distinct from $\calf$ and $\{[G]\}$.

There is a natural partial order on the set of free factor systems of $(G,\calf)$, where $\calf_1 \leq \calf_2$ if, for every $[A] \in \calf_1$, there exists $[B] \in \calf_2$ with $A \subseteq B$. Note that $\calf$ is minimal for this partial order. A \emph{free factor} of $(G,\calf)$ is an element of a free factor system.

\begin{defn}\label{Defn:Freefactorgraph}
    Let $(G,\calf)$ be a free product. The \emph{free factor graph of $(G,\calf)$}, denoted by $\mathrm{FF}(G,\calf)$, is the graph whose vertices are the proper free factors of $(G,\calf)$, two free factors $\calf_1$ and $\calf_2$ being adjacent if $\calf_1 < \calf_2$ or $\calf_2 < \calf_1$. 
\end{defn}

By results of Guirardel and Horbez~\cite[Proposition~2.11]{Guirardelhorbez19} (see also the work of Bestvina and Feighn~\cite{bestvina2014hyperbolicity} and Handel and Mosher~\cite{handel2014relative} for the case $G=F_N$), the graph $\mathrm{FF}(G,\calf)$ is Gromov hyperbolic.

We say that an outer automorphism $\Phi \in \Out(G)$ \emph{preserves} a free factor system $\calf$ if $\Phi$ fixes each conjugacy class in $\calf$. We write $\Out(G, \calf)$ to denote the subgroup of $\Out(G)$ which consists of the outer automorphisms which preserve $\calf$. The group $\Out(G,\calf)$ has a natural action by isometries on the graph $\mathrm{FF}(G,\calf)$ induced by its action on the set of free factors of $(G,\calf)$. The next result describes the loxodromic elements of $\mathrm{FF}(G,\calf)$. We say that an element $\Phi\in \Out(G,\calf)$ is \emph{fully irreducible} if no positive power of $\Phi$ preserves a proper free factor system of $(G,\calf)$. When $(G, \calf)$ is a free product and $\calf = \emptyset$, this reduces to the usual notion of being fully irreducible for elements of $\Out(F_N)$, where $F_N$ is a free group of rank $N$.

\begin{thm}\cite[Theorem~4.1]{Guirardelhorbez19}\label{Thm:loxofreefactorgraph} 
Let $(G,\calf)$ be a nonsporadic free product. An element $\Phi \in \Out(G,\calf)$ is a loxodromic element of $\mathrm{FF}(G,\calf)$ if and only if $\Phi$ is fully irreducible.
\end{thm}

The next theorem gives an existence condition of fully irreducible elements in subgroups of $\Out(G,\calf)$.

\begin{thm}\cite[Theorem~7.1]{Guirardelhorbez19}\label{Thm:fullyirreducibleexistence}
Let $(G,\calf)$ be a nonsporadic free product and let $H$ be a finitely generated subgroup of $\Out(G,\calf)$. If $H$ does not virtually preserve a proper $(G,\calf)$-free factor system then $H$ contains a fully irreducible outer automorphism.
\end{thm}

We now describe the Gromov boundary of $\mathrm{FF}(G,\calf)$. A \emph{$(G,\calf)$-tree} is an $\RR$-tree $T$ equipped with an action of $G$ by isometries such that, for every $i \in \{1,\ldots,k\}$, the group $G_i$ fixes a point in $T$. Given a $(G,\calf)$-tree $T$ and a point $x \in T$, we denote by $\Stab(x)$ the stabiliser of $x$. 

A $(G,\calf)$-tree is \emph{very small} if tripod stabilisers are trivial and arc stabilisers are cyclic (maybe trivial), nonperipheral and root-closed. 

A \emph{Grushko $(G,\calf)$-tree} is a $(G,\calf)$-tree $T$ such that $T$ is simplicial, the action of $G$ is minimal, edge stabilisers are trivial and, for every vertex $v$, either $\Stab(v)$ is trivial or $[\Stab(v)] \in \calf$. Recall that minimal means that $G$ does not preserve a proper subtree of $T$.

Note that, if $[A]$ is a free factor of $(G,\calf)$, the free factor system $\calf$ induces a free factor system $\calf|_A$ of $A$. A \emph{$(G,\calf)$-arational tree} is a very small $(G,\calf)$-tree $T$ which is not a Grushko tree and such that, for every free factor $[A]$ of $(G,\calf)$ the action of $(A,\calf|_A)$ on its minimal tree in $T$ induces a Grushko $(A,\calf|_A)$-tree. 

The following theorem relates the Gromov boundary of $\mathrm{FF}(G,\calf)$ with the $(G,\calf)$-arational trees.

\begin{prop}\cite[Theorem~3.4]{Guirardelhorbez19}\label{Prop:fixarationalboundary}
Let $(G,\calf)$ be a nonsporadic free product and let $H$ be a finitely generated subgroup of $\Out(G,\calf)$. If $H$ has a finite orbit in $\partial_{\infty} \mathrm{FF}(G,\calf)$, then $H$ has a finite index subgroup which fixes the homothety class of a $(G,\calf)$-arational tree.
\end{prop}

\begin{remark}
Note that \cite[Theorem~3.4]{Guirardelhorbez19} only shows that $H$ has a finite index subgroup preserving the $G$-equivariant homeomorphism class (for the observers' topology) of a $(G,\calf)$-arational tree $T$. However, by~\cite[Proposition~13.5]{guirardel2019algebraic}, the set of projective classes of $(G,\calf)$-arational trees which are equivalent to $T$ is a finite dimensional simplex. Therefore, the group $H$ has a finite index subgroup preserving the homothety class of a $(G,\calf)$-arational tree equivalent to $T$ corresponding to an extremal point of the simplex.
\end{remark}

If $T$ is a $(G,\calf)$-arational tree, we denote by $[T]$ its homothety class and by $\mathrm{SF} \colon \Stab([T]) \to \RR_+^*$ the stretching factor homomorphism. 

\begin{lemma}\cite[Proposition~6.3, Corollary~6.12]{Guirardelhorbez19}\label{Lem:stretchingfactor}
Let $T$ be a $(G,\calf)$-arational tree. For every $\Phi \in \Stab([T])$, we have $\mathrm{SF}(\Phi) \neq 1$ if and only if $\Phi$ is fully irreducible.
\end{lemma}  

\begin{lemma}\label{lem:polynonfullyirred}
    Let $(G,\calf)$ be a nonsporadic free product and let $\Phi \in \Out(G,\calf)$ be polynomially growing. Then $\Phi$ is not fully irreducible.
\end{lemma}
\begin{proof}
    Suppose towards a contradiction that $\Phi$ is fully irreducible. By \Cref{Thm:loxofreefactorgraph}, the element $\Phi$ is a loxodromic element of $\mathrm{FF}(G,\calf)$. In particular, $\Phi$ acts on $\mathrm{FF}(G,\calf)$ with North-South dynamics and has exactly two finite orbits in $\partial_{\infty}\mathrm{FF}(G,\calf)$ consisting of its attracting and repelling fixed points. Let $\xi_+$ be the attracting fixed point of $\Phi$. By \Cref{Prop:fixarationalboundary}, up to taking a power of $\Phi$, we may suppose that $\Phi$ fixes the homothety class of a $(G,\calf)$-arational tree $T$ associated with $\xi_+$. 
    By \Cref{Lem:stretchingfactor}, the stretching factor $\mathrm{SF}_T(\Phi)$ of $\Phi$ is distinct from $1$. Since $\mathrm{SF}_T \colon \langle \Phi \rangle \to \RR_+^*$ is a homomorphism, either $\mathrm{SF}_T(\Phi)>1$ or $\mathrm{SF}_T(\Phi^{-1})>1$. Up to replacing $\Phi$ by $\Phi^{-1}$, we may assume that $\lambda=\mathrm{SF}_T(\Phi)>1$. Note that we have only possibly replaced $\Phi$ by $\Phi^{-1}$ without changing the tree $T$. In particular, we are not considering the repelling fixed point $\xi_-$ of $\Phi$. This is why $\mathrm{SF}_T(\Phi) \circ \mathrm{SF}_T(\Phi^{-1})=1$.
    %Since $\xi_+$ is the attracting fixed point of $\Phi$, we have in fact $\lambda >1$.

    Let $g\in G$ and let $\ell([g])$ be the translation length of the conjugacy class of $g$ in $T$. Since $\Phi$ preserves the homothety class of $T$, for every $m \in \NN$, we have $\ell(\Phi^m([g]))=\lambda^m\ell([g])$. But $\ell(\Phi^m([g]))$ is bounded from above by a multiple of $\|\Phi^m([g])\|$ (see for instance~\cite[Propositions~1.5, 1.8]{CullerMorgan1987}). As $\Phi$ is polynomially growing, this implies that, for every $g \in G$, we have $\ell([g])=0$. Since $G$ is finitely generated, the group $G$ fixes a point in $T$ (see for instance~\cite[Section 3]{CullerMorgan1987}), a contradiction. Thus, $\Phi$ is not fully irreducible.
\end{proof}

\begin{remark}
    Another proof of the fact that $\mathrm{SF}_T(\Phi)>1$ is the following. We follow the notations of the above proof. 
    
    Suppose that $\Phi$ is fully irreducible and let $U$ be a relative train track associated with $\Phi$, which exists by \cite[Theorem~8.24]{francaviglia2015stretching}, \cite[Theorem~A]{lyman2022train}. Let $\lambda(\Phi)$ be the Perron--Frobenius eigenvalue of the associated transition matrix. A consequence of~\cite[Theorem~3.9]{francaviglia2021minimally} is that, since $(G,\calf)$ is nonsporadic and $\Phi$ is fully irreducible, we have $\lambda(\Phi)>1$. 
    
    Let $T_+=\lim_{n \to \infty} \frac{T\Phi^n}{\lambda(\Phi)^n}$.  By \cite[Lemma~2.14.1]{francaviglia2021action}, the space $T^+$ is an $\RR$-tree in the boundary of the outer space $\mathbb{P}\calo$ of $(G,\calf)$. Thus, $T_+\phi=\lambda(\Phi)T_+$. 
    
    Since $\Phi$ is fully irreducible, by \cite[Theorem~5.1.2]{francaviglia2021action}, it acts with a North-South dynamics on the outer space of $(G,\calf)$. By~\cite[Theorem~3]{Guirardelhorbez19}, there exists a $\Phi$-equivariant map $\psi \colon \mathbb{P}\calo \to \mathrm{FF}(G,\calf)$ with a $\Phi$-equivariant extension $\partial \psi$ from the set of $(G,\calf)$-arational trees to $\partial_\infty \mathrm{FF}(G,\calf)$. Since $\Phi$ also acts with a North-South dynamics on $\mathrm{FF}(G,\calf)$, we see that $\partial\psi$ sends $T_+$ to $\xi_+$, so that we can choose $T$ in the proof of Lemma~\ref{lem:polynonfullyirred} to be $T_+$. In particular, we have $\mathrm{SF}_T(\Phi)=\lambda(\Phi)>1$.\hfill$\qed$
\end{remark}

Bass--Serre theory implies that for any free product $(G,\calf)$, Grushko $(G,\calf)$-trees exist. In general there are many of these trees, even working up to equivariant homothety: they form a \emph{deformation space} which is invariant under an action of $\Aut(G, \calf)$, as studied in \cite{GuirardelLevittOuterSpace}. However, if $(G,\calf)$ is sporadic, then by \cite{ForesterRigidity} and \cite{LevittRigidity} there is a fixed point for this action --- a $(G,\calf)$-tree preserved by all of $\Aut(G,\calf)$. The tree in question has a single orbit of edges, and either one or two orbits of vertices (depending on the size of $\calf$). A more elementary proof of the same fact using translation length functions is in \cite[Proposition~4.8]{Andrew2021}: while the statement concerns indecomposable factors and the whole automorphism group, the argument applies just as well restricted to any sporadic free factor system and the subgroup $\Aut(G,\calf)$. 

The sporadic splittings obtained in the following propostion will be of the form $G_1 \ast \Z$ or $G_1 \ast G_2$ rather than the degenerate case $G=\Z$. However, note that the free $\Z$ action on a line is invariant under the (finite) automorphism group, so the statement holds even in this case.

%Recall that the Bass--Serre tree associated to a free product $(G, \calf)$ is a simplicial tree $T$ with a minimal $G$-action where the vertex stabilisers are of the form $G_i$ for $[G_i] \in \calf$, and edge stabilisers are trivial.

\begin{prop}\label{Prop:existenceBStree}
    Let $(G,\calf)$ be a free product and let $\Phi \in \Out(G,\calf)$ be polynomially growing. There is a free factor system $\calf'$ of $(G, \calf)$ and $k \in \NN$ such that $\Phi^k$ preserves $\calf'$ and $(G, \calf')$ is sporadic. In particular, $\Phi^k$ preserves a Bass--Serre tree associated to $\calf'$.
\end{prop}

\begin{proof}
    If $(G, \calf)$ is sporadic then the first claim is immediate. Otherwise, let $\calf'$ be a maximal $\Phi$-periodic proper free factor system. It suffices to prove that $(G, \calf')$ is sporadic. Up to taking a power of $\Phi$, we may suppose that $\calf'$ is $\Phi$-invariant. Thus, we may view $\Phi$ as an element of $\Out(G,\calf')$. By maximality of $\calf'$ and \Cref{Thm:fullyirreducibleexistence}, the element $\Phi$ is fully irreducible. Then, since $\Phi$ is polynomially growing, by \Cref{lem:polynonfullyirred} it must be the case that $(G, \calf')$ is sporadic.

    Now suppose that $(G, \calf)$ is sporadic and that $\Phi$ preserves $\calf$. Since there is a fixed point for the action of $\Aut(G,\calf)$ on its deformation space, this $(G,\calf)$-tree is preserved by $\Phi$. \qedhere
   
    %Now suppose that $(G, \calf)$ is sporadic and that $\Phi$ preserves $\calf$. Let $\phi$ be an automorphism in the outer class $\Phi$. By a classical result due to Culler and Morgan \cite{CullerMorgan1987}, the automorphism $\phi$ preserves the Bass--Serre tree $T$ associated to a splitting of $G$ exactly when it preserves the length function $l_T \colon G \to \mathbb{R}$ defined as \[l_T(g) = \inf\{d_T(x,g\cdot x) \mid x \in T\}.\]
    %In the case that $\calf = \{[G_1], [G_2]\}$ and $G = G_1 \ast G_2$, the vertex stabilisers of the associated Bass--Serre tree are conjugates of the $G_i$. Since $\phi$ maps each $G_i$ to its conjugate, it is clear that $\phi$ preserves the translation length function $l_T$. For the case when $\calf = \{[G_1]\}$ and $G = G_1 \ast \langle s \rangle$, where $s \in G$ has infinite order, the vertex stabilisers of the Bass--Serre tree $T$ are conjugates of $G_1$ and $s$ acts by translation. Moreover $\phi$ preserves the conjugacy class of $G_1$ and sends $s$ to an element of the form $a s b$, where $a, b \in G_1$. In particular, we again have that $\phi$ preserves the translation length function. Finally, the case where $\mathcal{F} = \emptyset$ and $G = \Z$ is trivial.
\end{proof}

\subsection{The cheap rebuilding property}\label{sec:CRP}

Let $\alpha \in \NN$. In this section, we give the relevant background regarding the \emph{cheap $\alpha$-rebuilding property}, which was introduced by Abert, Bergeron, Fraczyk and Gaboriau~\cite{AbertBergeronFraczykGaboriau2021} to prove that certain groups have vanishing (torsion) homology growth. Although we will not state the complete definition of this property, we list in the following propositions the properties which we will use in the rest of the paper. This property is relevant for our considerations by the following theorem.

\begin{thm}\cite[Theorem~10.20]{AbertBergeronFraczykGaboriau2021}\label{Thm:crptorsion}
 Let $\alpha \in \NN$ and let $\Gamma$ be a residually finite group of type $\mathrm{F}_{\alpha +1}$. Suppose that $\Gamma$ has the cheap $\alpha$-rebuilding property. For every Farber sequence $(\Gamma_k)_{k \in \mathbb N}$ of $\Gamma$, each $j \leq \alpha$ and every coefficient field $\KK$, we have
\[\lim_{k\to \infty} \frac{\dim_{\mathbb{K}} H_j (\Gamma_k, \mathbb{K})}{[\Gamma : \Gamma_k]} = 0 \quad \mbox{and} \quad \lim_{k\to \infty} \frac{\log | H_j(\Gamma_k,\mathbb Z)_{\rm tors}|}{[\Gamma:\Gamma_k]}=0.\]

\end{thm}

We refer to \cite{AbertBergeronFraczykGaboriau2021} for the definition of a Farber sequence. Examples include decreasing sequences of finite index normal subgroups with trivial intersection.

\begin{prop}\cite[Corollary~10.13]{AbertBergeronFraczykGaboriau2021}\label{prop:Crpprop}
Let $\Gamma$ be a residually finite countable group and let $\alpha \in \NN$. The following statements hold.
\begin{enumerate}
\item Let $\Gamma' \subseteq \Gamma$ be a finite index subgroup. Then $\Gamma$ has the cheap $\alpha$-rebuilding property if and only if $\Gamma'$ does.\label{prop:Crpprop:1}
\item If $\Gamma$ has an infinite normal subgroup $N$ such that $\Gamma/N$ is of type $F_\alpha$ and $N$ has the cheap
$\alpha$-rebuilding property, then $\Gamma$ has the cheap $\alpha$-rebuilding property. \label{prop:crpprop:2}
\item For every $m \in \NN^*$ and every $\alpha \in \NN$, the group $\ZZ^m$ has the cheap $\alpha$-rebuilding property.\label{prop:Crpprop:3}
\item If $\Gamma$ is infinite and virtually polycyclic then it has the cheap $\alpha$-rebuilding property.\label{prop:Crpprop:4}
\end{enumerate}
\end{prop}

Note that the ``infinite'' assumption at various points in this proposition was necessary: finite groups do not have the cheap $\alpha$-rebuilding property for any $\alpha$.

\begin{thm}\cite[Theorem~10.9]{AbertBergeronFraczykGaboriau2021}\label{thm:Crprebuild}
Let $\Gamma$ be a residually finite group acting on a CW-complex $\Omega$ in such a way that
any element stabilising a cell fixes it pointwise. Let $\alpha \in \NN$. Suppose that the following conditions hold:
\begin{enumerate}
\item $\Gamma\backslash\Omega$ has finite $\alpha$-skeleton;

\item $\Omega$ is $(\alpha-1)$-connected;

\item for each cell $\omega \in \Omega$ of dimension $j \le \alpha$ the stabiliser $\Stab_{\Gamma}(\omega)$ has the cheap $(\alpha-j)$-rebuilding property. 
\end{enumerate}
Then $\Gamma$ itself has the cheap $\alpha$-rebuilding property.
\end{thm}

\section{A combination theorem for the cheap rebuilding property of mapping tori}
\label{sec:combination}

Let $(G, \mathcal{F})$ be a free product. Recall that each element $[G_i] \in \mathcal{F}$ corresponds to the conjugacy class of a non-trivial finitely generated subgroup $G_i$ of $G$, with the possibility that $\mathcal{F} = \emptyset$. The main result of this section, \Cref{Thm:crpfreeproduct}, is a combination theorem which allows us to deduce the cheap rebuilding property for some mapping tori of $G$, assuming that it holds for the mapping tori of the factors $\mathcal{F}$. We will use the combination theorem in subsequent sections to prove the cheap $\alpha$-rebuilding property for a large family of mapping tori with polynomially-growing monodromy. 

\begin{thm}\label{Thm:crpfreeproduct}
    Let $G=G_1 \ast \ldots \ast G_k \ast F_N$ be a free product of residually finite groups. Fix $\alpha \in \NN$. Let $\phi$ be a polynomially-growing automorphism of $G$ which preserves the conjugacy classes of the factors $G_i$. Suppose that for every $i\in \{1,\ldots,k\}$, the group $G_i \rtimes_{\phi|_{G_i}} \ZZ$ has the cheap $\alpha$-rebuilding property. Then the group $G \rtimes_\phi \ZZ$ has the cheap $\alpha$-rebuilding property.
\end{thm}

    We note that the groups $G_i$ in the free product are not required to be freely irreducible. 

\begin{proof}
    The proof is by induction on the Grushko rank $k+N$ of $G$. If $k=1$ and $N=0$, then the group $G \rtimes_\phi \ZZ$ has the cheap $\alpha$-rebuilding property by the hypothesis. If $k=0$ and $N=1$, then $G \rtimes_\phi \ZZ$ is virtually $\ZZ^2$. By \Cref{prop:Crpprop}~\eqref{prop:Crpprop:1}~\eqref{prop:Crpprop:3}, for every $\alpha \in \NN$, the group $G \rtimes_\phi \ZZ$ has the cheap $\alpha$-rebuilding property.
    
    Suppose now that $k+N\ge 2$. Let $\Phi$ be the outer class of $\phi$ and let $\calf$ be a sporadic free factor system given by \Cref{Prop:existenceBStree}. Let $T$ be the canonical Bass--Serre tree of $G$ associated to $\calf$. The tree $T$ has a unique orbit of edges and its edge stabilisers in $G$ are trivial. In particular, vertex stabilisers in $G$ of $T$ are proper free factors of $G$, hence have a smaller Grushko rank than the one of $G$. Since $T$ is canonical, it is preserved by $\Phi$. Up to taking a power of $\Phi$ we may suppose that $\Phi$ acts trivially on the underlying graph of $G \backslash T$. 
    
    The actions of $G$ and $\Phi$ on $T$ induce an action of $G \rtimes_{\phi} \ZZ$ on $T$. Edge stabilisers in $G \rtimes_{\phi} \ZZ$ are infinite cyclic and the stabiliser of a vertex $v$ of $T$ is isomorphic to $\Stab(v) \rtimes_{\phi_v} \ZZ$, where $\Stab(v)$ is the vertex stabiliser of $v$ in $G$ and $\phi_v := \phi|_{\Stab(v)}$ is the automorphism of $\Stab(v)$ induced by a representative of $\Phi$ preserving $\Stab(v)$.

    We now prove that $G \rtimes_{\phi} \ZZ$ has the cheap $\alpha$-rebuilding property by applying \Cref{thm:Crprebuild} to the action of $G \rtimes_{\phi} \ZZ$ on $T$. Since the action is cocompact and since $T$ is a tree, it suffices to prove that the stabiliser of any vertex of $T$ has the cheap $\alpha$-rebuilding property and that the stabiliser of any edge of $T$ has the cheap $(\alpha-1)$-rebuilding property.
    
    Since edge stabilisers in $G \rtimes_{\phi} \ZZ$ are infinite cyclic, they have the cheap $(\alpha-1)$-rebuilding property by \Cref{prop:Crpprop}~\eqref{prop:Crpprop:3}. Let $v$ be a vertex of $T$. Note that, since $\phi$ is a polynomially-growing automorphism, so is $\phi_v \in \Aut(\Stab(v))$. Since the Grushko rank of $\Stab(v)$ is smaller than the one of $G$, by induction hypothesis, the group $\Stab(v) \rtimes_{\phi_v} \ZZ$ has the cheap $\alpha$-rebuilding property. Thus, by \Cref{thm:Crprebuild}, the group $G \rtimes_{\phi} \ZZ$ has the cheap $\alpha$-rebuilding property, which concludes the proof. 
\end{proof}

\begin{remark}
    \cref{Thm:crpfreeproduct} recovers the main result of \cite{AndrewHughesKudlinska2022} which states that any free-by-cyclic group $F_n \rtimes_{\phi} \Z$ with polynomially-growing monodromy $\phi \in \Aut(F_n)$ has the cheap $\alpha$-rebuilding property for every $\alpha \in \NN$.
\end{remark} 

\begin{remark}
If $G$ is infinitely-ended, residually finite and accessible then $G$ admits a finite index subgroup $H$ which is a free product of one-ended groups and a finitely generated free group, and the suspension $G \rtimes_\phi \ZZ$ has a finite index subgroup $H \rtimes_{\phi'} \ZZ$. Hence, by \Cref{prop:Crpprop}~\eqref{prop:Crpprop:1}, one can apply \cref{Thm:crpfreeproduct} in the setting of infinitely-ended, accessible groups which perhaps have torsion.
\end{remark}

\section{The cheap rebuilding property for residually finite (relatively) hyperbolic groups}
\label{sec:tf}

In this section,we prove the cheap rebuilding property for mapping tori of residually finite (relatively) hyperbolic groups. \Cref{Thm:crpfreeproduct} is the main step in order to prove the cheap rebuilding property for infinitely-ended hyperbolic groups. The one-ended case requires the use of the JSJ decomposition of the group, whose properties are presented after \Cref{lem:surfacegroupcrp}, following Guirardel--Levitt~\cite{guirardellevitt2017jsj}.

\begin{lemma}\label{lem:surfacegroupcrp}
    Let $S$ be a compact, connected hyperbolic surface and let $G=\pi_1(S)$. Let $D \in \MCG(S)$ be a Dehn multi-twist. For every $\alpha \in \NN$, the group $G \rtimes_D \ZZ$ has the cheap $\alpha$-rebuilding property.
\end{lemma}

\begin{proof}
    Let $\beta=(\beta_1,\ldots,\beta_k)$ be the collection of pairwise non-isotopic simple closed curves in $S$ associated to the multi-twist $D$. Then $D$ preserves each connected component of the complement $S - \bigcup_{i = 1}^k \beta_i$. Hence, the natural action of $G$ on the simplicial tree $T$ dual to the collection of lifts of the loops $\beta_i$ in the universal cover of $S$, extends to an action of the extension $G \rtimes_{D} \Z$. The edge stabilisers of the action are isomorphic to $\Z^2$ and the vertex stabilisers are isomorphic to $\pi_1(S_0) \times \Z$, where $S_0$ is a connected component of the closure of $S - \bigcup_{i = 1}^k \beta_i$. In particular, by \Cref{prop:Crpprop}~\eqref{prop:crpprop:2} and~\eqref{prop:Crpprop:3}, the edge and vertex stabilisers satisfy the cheap $\alpha$-rebuilding property for every $\alpha$. Thus, by \Cref{thm:Crprebuild}, the group $G \rtimes_D \ZZ$ has the cheap $\alpha$-rebuilding property for every $\alpha \in \NN$. 
\end{proof}

Let $\mathcal{G}$ be the family of groups which are hyperbolic relative to a finite family of virtually polycyclic groups. For each $G \in \mathcal{G}$ we may assume that the peripheral family $\mathcal{P}$ does not contain virtually cyclic groups, since removing them does not destroy relative hyperbolicity (see \cite[Section~9.3]{guirardellevitt2017jsj}). The family $\mathcal{G}$ includes the class of toral relatively hyperbolic groups, which are torsion-free groups hyperbolic relative to a finite collection $\calp$ of finitely generated abelian subgroups. 

We say a group $G$ is \emph{one-ended relative to a collection of subgroups $\calp$} if there does not exist a splitting of $G$ over finite subgroups such that each group in $\calp$ is conjugate into some vertex group. Note that a one-ended group is automatically one-ended relative to every collection of subgroups. We write $\Aut(G; \mathcal{P})$ to denote the group of automorphisms of $G$ which preserve $\mathcal{P}$. 

Let $G \in \mathcal{G}$ and suppose that $G$ is one-ended relative to $\calp$. By the work of Guirardel and Levitt~\cite[Corollary~9.20]{guirardellevitt2017jsj} (see also~\cite[Section~3.3]{GuirardelLevitt2015}), there is a canonical JSJ tree $T_G$ for $G$, that is, a simplicial tree equipped with an action of $G$ which is preserved by the elements of $\Aut(G; \mathcal{P})$. The group $\Aut(G; \mathcal{P})$ has a finite index subgroup $\calk(T_G)$ which acts as the identity on the underlying graph of $G \backslash T_G$. Edge stabilisers of $T_G$ in $G$ are virtually polycyclic. If we further assume that $G$ is torsion free then the stabiliser $\Stab_G(v)$ of a vertex $v$ in $T_G$ satisfies one of the following:
\begin{enumerate}
    \item the group $\Stab_G(v)$ is isomorphic to the fundamental group of a compact hyperbolic surface $S$ and the image of the natural homomorphism $\calk(T_G) \to \Out(\Stab_G(v))$ is contained in the mapping class group $\MCG(S)$ of $S$;
    \item there exists $[P] \in \calp$ with $\Stab_G(v)=P$. In particular, the group $\Stab_G(v)$ is virtually polycyclic;
    \item the image of the natural homomorphism  $\calk(T_G) \to \Out(\Stab_G(v))$ is finite.
\end{enumerate}

Recall that $\mathcal{G}$ denotes the family of groups which are hyperbolic relative to a finite family of virtually polycyclic groups, and we assume with no loss of generality that the peripheral subgroup are not virtually cyclic. Let $\mathcal{G}_{rf}$ denote the subset of the groups in $\mathcal{G}$ which are  residually finite.

\begin{lemma}\label{lemma:vtorsionfree}
    For any group $G \in \mathcal{G}_{rf}$, there exists a finite-index subgroup $G' \leq G$ such that $G' \in \mathcal{G}_{rf}$ and $G'$ is torsion free.
\end{lemma}

\begin{proof}
    It is well known, using for instance the action of $G$ on a relative Rips complex (see~\cite{dahmani2003classifying}), that for a relatively hyperbolic group $G$, there exists a finite number of finite subgroups $F_1, \ldots, F_k$ of $G$, such that any finite-order element $g \in G$ can be conjugated into some $F_i$, or into a peripheral subgroup of $G$. Since polycyclic groups contain finitely many conjugacy classes of finite order elements by \cite[Theorem~7.1]{baumslag1991algorithmic}, it follows that $G$ also has finitely many conjugacy classes of finite order elements. Now since $G$ is residually finite, there exists a finite index normal subgroup $G'$ of $G$ which does not contain any of the finite conjugacy classes, and thus is torsion free. Moreover, $G'$ is clearly also residually finite and hyperbolic relative to a finite collection of virtually polycyclic subgroups.
\end{proof}

\begin{thm}\label{thm:hypgroupcrp}
     Let $G \in \mathcal{G}_{rf}$ and let $\Phi=[\phi] \in \Out(G)$ be a polynomially-growing outer automorphism. Then $G \rtimes_{\phi} \ZZ$ has the cheap $\alpha$-rebuilding property for all $\alpha$.
\end{thm}

\begin{proof}
   By \cref{lemma:vtorsionfree}, we may assume that $G$ is torsion free. By \Cref{Thm:crpfreeproduct}, it suffices to prove the result when $G$ is one ended. Let $T_G$ be the JSJ tree associated to $G$ described above. By \cite[Lemma~3.2]{MinasyanOsin2012}, we may pass to a power of $\Phi$ which fixes the conjugacy class of each group in the peripheral system $\mathcal{P}$. That is, up to finite index (applying \Cref{prop:Crpprop}~\eqref{prop:Crpprop:1}) we can assume $\phi$ is an element of $\Aut(G;\calp)$.

   Since $T_G$ is preserved by $\Aut(G; \mathcal{P})$, the action of $G$ and $\phi$ on $T_G$ induces an action of $G \rtimes_{\phi} \ZZ$ on $T_G$. Up to taking a power of $\phi$, we may suppose that $\phi \in \calk(T_G)$.

   We prove \Cref{thm:hypgroupcrp} by applying \Cref{thm:Crprebuild} to the action of $G \rtimes_{\phi} \ZZ$ on $T_G$. As in the proof of \Cref{Thm:crpfreeproduct}, it suffices to prove that, for every cell $\omega \in T_G$, the group $\Stab(\omega)$ has the cheap $\alpha$-rebuilding property for every $\alpha \in \NN$. 

   Edge stabilisers in $G \rtimes_{\phi} \ZZ$ are virtually polycyclic and infinite since $G \rtimes_{\phi} \ZZ$ is one-ended. Thus, by \Cref{prop:Crpprop}~\eqref{prop:Crpprop:4}, they have the cheap $\alpha$-rebuilding property for every $\alpha \in \NN$.  
   
   Let $v \in T_G$.  There are three cases for the vertex stabilisers.

    \paragraph{\underline{\textbf{Case 1:}}} \emph{$\Stab_G(v)$ is the fundamental group of a compact hyperbolic surface $S$ and $\calk(T_G)\to\Out(\Stab_G(v))$ has image contained in $\MCG(S)$.}

    Since $\phi$ is polynomially growing, its image in $\MCG(S)$ is in fact a Dehn multi-twist $D$. Thus, the stabiliser in $G \rtimes_{\phi} \ZZ$ of $v$ is isomorphic to $\pi_1(S) \rtimes_{D} \ZZ$. By \Cref{lem:surfacegroupcrp}, it has the cheap $\alpha$-rebuilding property for every $\alpha \in \NN$. \hfill$\blackdiamond$
   
   \paragraph{\underline{\textbf{Case 2:}}}  \emph{$\Stab_G(v)$ is virtually polycyclic.}
   
   Here the stabiliser of $v$ in $G \rtimes_{\phi} \ZZ$ is virtually polycyclic-by-$\Z$ (hence, polycyclic).  Thus, it has the cheap $\alpha$-rebuilding property for every $\alpha \in \NN$ by \Cref{prop:Crpprop}~\eqref{prop:Crpprop:4}. \hfill$\blackdiamond$
   
   \paragraph{\underline{\textbf{Case 3:}}}  \emph{The image of the natural homomorphism $\calk(T_G)\to \Out(\Stab_G(v))$ is finite.}

   Up to taking a power of $\phi$, we may suppose that the stabiliser of $v$ in $G \rtimes_{\phi} \ZZ$ is isomorphic to $\Stab_G(v) \times \ZZ$. By for instance~\cite[Lemma~3.8]{GuirardelLevitt2015}, the group $\Stab_G(v)$ belongs to $\mathcal{G}$. By results of Dahmani~\cite[Theorem~0.1]{dahmani2003classifying}, the group $\Stab_G(v)$ is of type $F_\infty$. By \Cref{prop:Crpprop}~\eqref{prop:crpprop:2}, for every $\alpha \in \NN$, the group $\Stab_G(v) \times \ZZ$ has the cheap $\alpha$-rebuilding property. \hfill$\blackdiamond$

   Thus, \Cref{thm:Crprebuild} implies that $G \rtimes_\phi \ZZ$ has the cheap $\alpha$-rebuilding property for every $\alpha \in \NN$. 
\end{proof}

\begin{remark}
    The restriction to virtually polycyclic subgroups is also used in Case (3); to deduce that $\Stab_G(v)$ is of type $F_\infty$ requires that arbitrary subgroups in $\calp$ themselves are of finite type.
\end{remark}

\section{Right angled Artin and Coxeter groups} \label{sec RA*Gs}

Let $L$ be a flag complex and recall that $A_L$ and $W_L$, respectively, denote the \emph{right-angled Artin} and \emph{right-angled Coxeter groups} on $L$. The following maps induce automorphisms of both the right-angled Artin group $A(L)$ and the right-angled Coxter group $W(L)$:
\begin{enumerate}
    \item \emph{graph automorphisms}, which are the automorphisms induced by automorphisms of $L$;
    \item \emph{inversions}, which send $v\mapsto v^{-1}$ and $u\mapsto u$ for $u\neq v$ and $u,v\in L^{(0)}$;
    \item \emph{partial conjugations}, labelled by $k_{w,C}$ for $w\in L^{(0)}$ and a connected component $C$ of $L\backslash \st(w)$, and defined as $k_{w,C}(u) = w^{-1}uw$ if $u\in C^{(0)}$ and $k_{w,C}(u) =u$ if $u\in L^{(0)}\backslash C$; and
    \item \emph{folds}, labelled by $\tau_{v,w}$ for any $v,w \in L^{(0)}$ with $\lk v\subseteq \lk w$, and defined by $\tau_{v,w}(v) =vw$ and $\tau_{v,w}(u) =u$ for all $u\in L^{(0)}\backslash\{v\}$.
\end{enumerate}
We say an automorphism of $A_L$ is \emph{untwisted} if it is contained in the subgroup $U(L)\leqslant\Aut(A_L)$ which is generated by graph automorphisms, inversions, partial conjugations and folds. We define the subgroup of untwisted automorphisms of the right angled Coxeter groups analogously. Note that by \cite[Proposition~A(3)]{Fioravanti22} untwisted automorphisms of $A_L$ are exactly the automorphisms which preserve the standard coarse median structure on $A_L$.  We will not use this fact or any results about coarse medians explicitly but we note that it underpins much of our work in this section.

\begin{thm}\label{thm:RAAGs:rebuild}
    Let $L$ be a flag complex on $[m]$ and let $\Gamma=A_L\rtimes_\phi \Z$.  If $\phi$ is an untwisted and polynomially-growing automorphism of $A_L$, then $\Gamma$ has the cheap $\alpha$-rebuilding property for all $\alpha$.
\end{thm}
\begin{proof}
We proceed by induction on $m$, the number of vertices of $L$.  When $m=1$ we have that $A_L$ is isomorphic to $\Z$.  In this case $\Gamma$ is virtually $\Z^2$ and so the result follows from \Cref{prop:Crpprop}. We now suppose $m>1$.  Note that if $K\subset L$ is a full subcomplex then any untwisted automorphism of $A_L$ preserving $A_K$ restricts to an untwisted automorphism of $A_K$.  To prove the inductive step there are three cases to consider.
    
    \paragraph{\underline{\textbf{Case 1:}}} \emph{$A_L$ is freely reducible.}
    In this case $A_L$ admits a Grushko splitting $A_{K_1}\ast\dots\ast A_{K_k}\ast F_n$ where each $K_i$ and $[n]$ is a full subcomplex of $L$. In particular, each  subcomplex $K_i$ contains at least one vertex but strictly less than $m$ vertices. We replace $\phi$ by a sufficiently high power which preserves the conjugacy class of every factor $A_{K_i}$. By the inductive hypothesis each $A_{K_i}\rtimes_{\phi|_{A_{K_i}}}\Z$ has the cheap $\alpha$-rebuilding property for all $\alpha \in \NN$.  The conclusion follows from \Cref{Thm:crpfreeproduct}.\hfill$\blackdiamond$
    
    \paragraph{\underline{\textbf{Case 2:}}} \emph{$A_L$ is freely and directly irreducible.}
    By \cite[Proposition~D]{Fioravanti22} the group $A_L$ splits as an amalgamated free product $A_J\ast_{A_{J\cap K}}A_K$ with each $J,K,J\cap K\subset L$ non-empty and such that, possibly after replacing $\phi$ by a high enough power, the corresponding Bass--Serre tree $T$ is $\phi$-invariant. Moreover, for $X\in\{J,K,J\cap K\}$ we see $\phi(A_X)=A_X$ (see \cite[Lemma~5.3]{Fioravanti22}).  It follows $\Gamma$ acts on $T$ with stabilisers conjugate to $A_X\rtimes_{\phi|_{A_X}}\Z$.  By the inductive hypothesis, the subgroups $A_X\rtimes_{\phi|_{A_X}}\Z$ satisfy the cheap $\alpha$-rebuilding property for all $\alpha\in\NN$.  The conclusion follows from \Cref{thm:Crprebuild}.\hfill$\blackdiamond$
    
     \paragraph{\underline{\textbf{Case 3:}}} \emph{$A_L$ is directly reducible.}
    
    Now, $A_L$ splits as a direct product $\prod_{i} A_{K_i} \times \Z^k$ for some $k \geq 0$, where each $A_{K_i}$ is non cyclic and directly irreducible. If $k > 0$ then since $\prod_{i} A_{K_i}$ is of type $F_{\alpha}$ for every $\alpha \in \NN$, we have that $A_L$ has the cheap $\alpha$-rebuilding property for every $\alpha$ by \cref{prop:Crpprop}~\eqref{prop:crpprop:2}. Another application of \cref{prop:Crpprop}~\eqref{prop:crpprop:2} gives us that the mapping torus $A_L \rtimes_{\phi} \Z$ has the cheap $\alpha$-rebuilding property for every $\alpha \in \NN$ and $\phi \in \mathrm{Aut}(A_L)$.
    
    Hence, we may assume that $k = 0$. Thus, $A_L$ acts on the product $X=\prod_i T_i$ where each $T_i$ is a tree which arises from \Cref{Prop:existenceBStree} if $A_{K_i}$ is freely reducible, and from the amalgamated product splitting given by \cite[Proposition~D]{Fioravanti22} as in Case 2 when $A_{K_i}$ is both directly and freely irreducible.
    
    As in the previous two cases, up to passing to a power we may assume that $\phi$ fixes the quotient $A_L \backslash X$ pointwise, and that $\phi|_{A_{K_i}}$ preserves the stabilisers of $A_{K_i}$ acting on $T_i$ for each $i$.  Since the stabilisers of $A_L$ acting on $X$ are products of the stabilisers of the $A_{K_i}$ acting on $T_i$ we see that $\phi(\Stab_{A_L}(\sigma))=\Stab_{A_L}(\sigma)$ for each cell $\sigma\in X$.  Thus, $A_L\rtimes_\phi\Z$ acts on $X$ with the stabiliser of a cell $\sigma$ isomorphic to $A_{J_\sigma}\rtimes_{\phi|_{A_{J_\sigma}}}\Z$ for some RAAG $A_{J_\sigma}$ where $J_\sigma$ is non-empty and has strictly less vertices than $L$.  
    The conclusion follows from \Cref{thm:Crprebuild}.\hfill$\blackdiamond$
    
    This completes the proof of the inductive step and the theorem.
\end{proof}

\begin{thm}\label{thm:RACGs:rebuild}
    Let $L$ be a flag complex on $[m]$ and let $\Gamma=W_L\rtimes_\phi \Z$.  If $\phi$ is polynomially growing, then $\Gamma$ has the cheap $\alpha$-rebuilding property for all $\alpha$.
\end{thm}
\begin{proof}
   The proof is entirely analogous to \Cref{thm:RAAGs:rebuild} with the following modifications. First, we note that the subgroup of untwisted automorphisms of $W_L$ has finite index in the automorphism group $\Aut(W_L)$ by \cite{SaleSusse2019}. Thus, we may pass to a power of $\phi$ which is untwisted.
   
   In the case where $m=1$ we have that $\Gamma=\Z/2\times \Z$ which has the cheap $\alpha$-rebuilding property for all $\alpha\in\NN$ by \Cref{prop:Crpprop}. The three cases are now identical, taking into account the remarks after Theorem~E and at the start of Section~5 of \cite{Fioravanti22}, since the results we use for $A_L$ also hold for $W_L$.
\end{proof}

\begin{remark}
    We can actually say more regarding the cheap $\alpha$-rebuilding property for mapping tori of automorphisms of RAAGs.  Indeed, if $L$ is $(\alpha-1)$-connected, then $A_L$ has the cheap $\alpha$-rebuilding property \cite[Theorem~I]{AbertBergeronFraczykGaboriau2021}.  
In particular, if $L$ is contractible then for any automorphism $\phi$ of $A_L$, the group $A_L\rtimes_\phi\Z$ has the cheap $\alpha$-rebuilding property for all $\alpha$. 
\end{remark} 
%On the other hand, by \cite{AvramidiOkunSchreve2021,OkunSchreve2021}, the mod-$p$ torsion homology growth of $A_L$ equals the reduced mod-$p$ Betti numbers of $L$ shifted by a degree.  
%In turn, by \cite[Theorem~B]{FisherHughesLeary2022}, this is equal to the dimension of the homology of $A_L$ with coefficients in a certain universal division ring $\cald_{\FF_p A_L}$ (sometimes called agrarian homology).  
%Suppose $\widetilde H^\ast(L;\QQ)=0$ and $\widetilde H^\ast(L;\FF_p)\neq 0$.  Let $\Gamma=A_L\rtimes_\phi\Z$ and suppose $\cald_{\FF_p \Gamma}$ exists (this will be true if $\Gamma$ is residually finite rationally solvable for instance).  In this case the $\ell^2$-torsion of $\Gamma$ will vanish by \cite{DavisLeary2003} and \cite[Theorem~7.27(7)]{Lueck2002}.
%But, in \cite[Section~4]{HennekeKielak2021} the authors introduce agrarian torsion $\tau^{\cald_{\FF_p \Gamma}}(\Gamma)$ taking values in $(\cald_{\FF_p \Gamma})^\times_{\mathrm{ab}}/\{\pm1\}$.  In light of this we raise the following question.

%\begin{question}
%Let $L$ be a flag complex on $[m]$ such that $\widetilde H^\ast(L;\QQ)=0$ and $\widetilde H^\ast(L;\FF_p)\neq 0$ for some prime $p$.  Let $\Gamma=A_L\rtimes_\phi\Z$ where $\phi$ is exponentially growing.  Does $\cald_{\FF_p \Gamma}$ exist and if so is $\tau^{\cald_{\FF_p \Gamma}}(\Gamma)\neq 0$?
%\end{question}

\bibliographystyle{alpha}
\bibliography{refs.bib}

\end{document}